\documentclass[12pt, reqno]{amsart}
\usepackage{mathrsfs}
\usepackage{amsmath, amstext, amsbsy, amssymb, amscd,mathrsfs}
\usepackage{amsmath}
\usepackage{amsxtra}
\usepackage{amscd}
\usepackage{amsthm}
\usepackage{amsfonts}
\usepackage{amssymb}
\usepackage{eucal}
\usepackage{color}

\usepackage[active]{srcltx}

\input postpictex
\input xy

\xyoption{all}

\newtheorem{theorem}{Theorem}[section]
\newtheorem{lemma}[theorem]{Lemma}
\newtheorem{proposition}[theorem]{Proposition}

\theoremstyle{definition}

\newtheorem{remark}[theorem]{Remark}

\numberwithin{equation}{section}

\newcommand{\g}{\mathfrak{g}}

\newcommand{\m}{\mathfrak{m}}
\newcommand{\q}{\mathfrak{q}}

\newcommand{\ra}{\rightarrow}
\newcommand{\sudim}{\underline{\text{dim}}\,}
\newcommand{\ev}[1]{{#1}_{\bar{0}}}
\newcommand{\od}[1]{{#1}_{\bar{1}}}

\newcommand{\la}{\lambda}

\newcommand{\ad}{\text{ad}\,}

\newcommand{\gl}{{\mathfrak{gl}}}
\newcommand{\salg}{{\mathfrak{salg}}}

\newcommand{\Z}{ \mathbb Z }

\newcommand{\wh}{\text{Wh}}
\newcommand{\whmod}{\g\texttt{-Wmod}^\chi}
\newcommand{\Wmod}{ \mathcal{W}_\chi\texttt{-mod}}

\newdimen\Hoogte    \Hoogte=12pt
\newdimen\Breedte   \Breedte=12pt
\newdimen\Dikte     \Dikte=0.5pt
\newenvironment{Young}{\begingroup
       \def\vr{\vrule height0.8\Hoogte width\Dikte depth 0.2\Hoogte}
       \def\fbox##1{\vbox{\offinterlineskip
                    \hrule height\Dikte
                    \hbox to \Breedte{\vr\hfill##1\hfill\vr}
                    \hrule height\Dikte}}
       \vbox\bgroup \offinterlineskip \tabskip=-\Dikte \lineskip=-\Dikte
            \halign\bgroup &\fbox{##\unskip}\unskip  \crcr }
       {\egroup\egroup\endgroup}
\def\Diagram#1{\relax\ifmmode\vcenter{\,\begin{Young}#1\end{Young}\,}\else%
              $\vcenter{\,\begin{Young}#1\end{Young}\,}$\fi}
{\vskip-\lastskip\medskip
  \noindent
  {\em #1.}\enspace
  }%
{\qed\par\medskip
  }

\makeatletter
\def\Ddots{\mathinner{\mkern1mu\raise\p@
\vbox{\kern7\p@\hbox{.}}\mkern2mu
\raise4\p@\hbox{.}\mkern2mu\raise7\p@\hbox{.}\mkern1mu}}
\makeatother

\begin{document}
\title[$W$-superalgebras of type $Q$]
{Finite $W$-superalgebras for queer Lie superalgebras}

\author[Lei Zhao]{Lei Zhao}
\address{Department of Mathematics, University of Oklahoma,
Norman, OK 73019} \email{prescheme@ou.edu; prescheme@gmail.com}

\begin{abstract}
We initiate and develop the theory of finite $W$-superalgebras $\mathcal{W}_\chi$ associated to the queer Lie superalgebra $\g=\q(N)$ and a nilpotent linear functional $\chi \in \ev\g^*$. We show that the definition of the $W$-superalgebra is independent of various choices. We also establish a Skryabin type equivalence between the category of $\mathcal{W}_\chi$-modules and a category of certain $\g$-modules.
\end{abstract}

\maketitle
\date{}
  \setcounter{tocdepth}{1}
  \tableofcontents

\section{Introduction}\label{sec:intro}

\subsection{}
The finite $W$-algebras are certain associative algebras associated
to a complex semisimple Lie algebra and a nilpotent element in it.
The study of finite $W$-algebras in special cases dates back to
Lynch's thesis \cite{Ly}, which is in turn a generalization of
Kostant's construction in the regular nilpotent case [Kos]. In full
generality, the finite $W$-algebras were introduced by Premet
\cite{Pr1,Pr2}. In his course to prove the celebrated Kac-Weisfeiler
conjecture, Premet constructed a modular version of finite
$W$-algebras in \cite{Pr1}. He then made the transition in
\cite{Pr2} to define the characteristic zero version of finite
$W$-algebras and showed that they are noncommutative deformations of
the coordinate rings of Slodowy slides \cite{Slo}. Finite $W$-algebras are a very active area of research, we refer the reader to review papers by Wang \cite{Wan} and Losev \cite{Los} on this subject and references therein.

In \cite{WZ1,WZ2}, Wang and the author have developed the
representation theory of Lie superalgebras in prime characteristic
by formulating and proving a super analogue of the Kac-Weisfeiler
conjecture for basic classical Lie superalgebras or a queer Lie
superalgebra. In the course of the proof, the modular finite $W$-superalgebras
were first introduced as a super-generalization of \cite{Pr1}.

The main goal of this paper is to initiate and develop the
theory of finite $W$-superalgebras for the queer Lie superalgebra
over the field of complex numbers. For semisimple Lie algebras, Gan
and Ginzburg provided in \cite{GG} a purely characteristic zero
approach in a generalized form of Premet's result \cite{Pr2} which
allows us to regard finite $W$-algebras as quantizations of Slodowy
slides. Our approach will be a superalgebra generalization of
\cite{GG}. Let us explain the paper in some detail.

%The representation theory of finite $W$-algebras has been developed
%most adequately in type $A$ in a series of papers by Brundan and
%Kleshchev (see e.g. \cite{BK2,BK3}). In particular, they established
%a remarkable higher level duality of finite $W$-algebras of type $A$
%and the cyclotomic Hecke algebras, which recovers the classical
%Schur-Weyl duality at level one \cite{BK3}. On the other hand,
%Sergeev \cite{Ser} extended the classical Schur-Weyl duality to a
%duality between the twisted tensor product
%$\mathcal{S}(d)$ of the Clifford algebra and the group algebra of the symmetric group and the queer
%Lie superalgebra $\g$ on $V^{\otimes d}$, where $V$ denotes the
%natural $\g$-module. The second main goal of this paper is to
%develop the higher Sergeev duality along the line of \cite{BK3}.

\subsection{}
Recall that the queer Lie superalgebra $\g=\q(N)$ has even part
$\ev\g \cong \gl(N)$ and its odd part $\od\g$ is another copy of
$\gl(N)$ under the adjoint action of $\ev\g$. Let $\Pi:\g\ra\g$ be the odd involution which interchanges these copies of $\mathfrak{gl}(N)$ in the obvious way. Let $\chi \in \ev\g^*$ be a nilpotent linear
functional, regard $\chi\in \g^*$ by letting $\chi(\od\g)=0$.
Note that $\g$ admits an {\em odd} nondegenerate invariant bilinear form $(-,-)$. Then the element $E
\in \od\g$ determined by $\chi=(E,-)$ is an odd nilpotent element in $\g$. Set $e=\Pi(E)$.

We start out by defining and classifying all the {\em good} $\Z$-gradings
of $\g$ for $\chi$. (See \cite{FRTW,KRW,EK} for the Lie algebra case.) It
turns out that (Lemma~\ref{lem:grading-good-gl-q}) a $\Z$-grading
$\Gamma: \g =\oplus_{j\in\Z}\g_j$ is good if and only if when
restricted to $\ev\g$, $\Gamma$ is good for $e$.

We fix a good grading $\Gamma: \g=\oplus_{j\in\Z}\g_j$. Note that $\g_{-1}$ has both even and odd components. There is an even non-degenerate skew-supersymmetric bilinear
form
\begin{equation*}%\label{ssform}
\langle x, y \rangle := (E, [x, y]) = \chi([x, y])
\end{equation*}
on $\g_{-1}$. We pick an isotropic subspace $\mathfrak{l}$ in $\mathfrak{g}_{-1}$ and then define the finite $W$-superalgebra $W_{\mathfrak{l}}$ associated to $\mathfrak{l}$ following
\cite{GG} (which is in turn a generalization of \cite{Pr2}).

It follows from \cite[Lemma~1.1]{EK} that there exist
$h\in\g_{0,\bar{0}}$ and $f\in\g_{-2,\bar{0}}$ such that $\{e,h,f\}$
form an $\mathfrak{sl}(2)$-triple. Put $F=\Pi(f)$. We define the
Slodowy slide $\mathcal{S}$ of $\g$ through $\chi$ to be the affine
superscheme $\chi+\text{ker}\,\text{ad}^*F \subseteq \g^*$, where
$\text{ad}^*$ denotes the coadjoint action of $\g$ on $\g^*$.
Introduce the Kazhdan filtration on $\mathcal{W}_{\mathfrak{l}}$ following \cite{GG}. Denote
the associated graded superalgebra by $\text{gr}^K\mathcal{W}_{\mathfrak{l}}$. Following \cite{GG}, we are able to show that there is
a canonical isomorphism (Theorem~\ref{thm:isom-nu})
\[
\nu: \text{gr}^K\mathcal{W}_{\mathfrak{l}} \cong \mathbb{C}[\mathcal{S}].
\]
To establish this isomorphism we need some cohomology vanishing result for Lie superalgebras, which relies on the acyclicity of the super Koszul complex. It follows that different choices of isotropic subspaces
$\mathfrak{l}$ give rise to isomorphic $W$-superalgebras
$\mathcal{W}_{\mathfrak{l}}$. Combining with the classification of
good grading for $\chi$ and work of Brundan-Goodwin \cite{BG}, we
conclude that the definition of $W$-superalgebra is also independent
of the good gradings. We may fix a Lagrangian subspace
$\mathfrak{l}$ and use
$\mathcal{W}_\chi$ for $W_{\mathfrak{l}}$ without causing any
confusion.

Let $\Wmod$ be the category of finitely generated
$\mathcal{W}_\chi$-modules and denote $\whmod$ the category of
finitely generated $\g$-modules upon which $(x-\chi(x))$ act locally
nilpotently for all $x\in \m$. If $M \in \whmod$, then the subspace
\[
\text{Wh}(M)=\{v\in M \vert\; (x-\chi(x))v=0,\forall x\in \m\},
\]
is a $\mathcal{W}_\chi$-module, hence $M \mapsto \text{Wh}(M)$ is a
functor from $\whmod$ to $\Wmod$. Also, we have the functor
$Q_\chi\otimes_{\mathcal{W}_\chi}-$ from $\Wmod$ to $\whmod$. Using
the approach of \cite{GG}, we are able to establish
(Theorem~\ref{thm:skry}) a Skryabin type equivalence (see
\cite{Skr}) for $\g$, that is, the functors $\text{Wh}(-)$ and
$Q_\chi\otimes_{\mathcal{W}_\chi}-$ are quasi-inverse equivalences
between $\whmod$ and $\Wmod$. We will also call this the Skryabin
equivalence (for $\g$).

We remark that the original approach to Skryabin equivalence by Skryabin \cite{Skr} superizes without difficulties here, which allows the ``finite generation'' assumption to be removed.

%Let us remark that for a basic classical Lie superalgebra $\mathfrak{s}$ with a good grading for a nilpotent linear character $\xi \in %\ev{\mathfrak{s}}^*$, if the odd part of the degree $-1$ component of $\mathfrak{s}$ is even dimensional, then our method also applies, that is, we %can prove that the definition of the $W$-superalgebra is independent of the choice of the isotropic subspaces and can establish the Skryabin %equivalence for $\mathfrak{s}$.

\subsection{}
The paper is organized as follows. In Section~\ref{sec:basics} we classify good $\Z$-gradings
for $\g$. In Section~\ref{sec:defW}, we
prove that the definition of the finite $W$-superalgebra does not
depend on the choice of isotropic subspaces or good gradings.

Throughout we work with the field $\mathbb{C}$ of complex numbers as
the ground field.

For a superspace (i.e. $\Z_2$-graded vector space) $M = \ev{M} \oplus
\od{M}$, write $|v| \in \Z_2$ for the parity (or degree) of $v
\in M$, which is implicitly assumed to be $\Z_2$-homogeneous. The
graded dimension of $M$ will be
denoted by $\sudim M = \dim \ev M \vert \dim \od M$.

By vector spaces, derivations, subalgebras, ideals, modules,
submodules, and commutativity, etc. we mean in the super sense
unless otherwise specified.

\noindent {\bf Acknowledgments.} The author is very grateful to Weiqiang Wang for suggesting the problem as well as
offering valuable advice. He thanks Jon Kujawa for helpful discussions. He also thanks the anonymous referee for valuable suggestions.

\section{Preliminaries}\label{sec:basics}
\subsection{Algebraic supergroup $Q(N)$ and the queer Lie superalgebra $\q(N)$}
We use the usual functorial language for superschemes and supergroups as in e.g. \cite[Section~2]{BK}.

The algebraic supergroup $G=Q(N)$ is the functor which associates to
any $A$ in the category $\salg$ of commutative $\mathbb{C}$-superalgebras with even homomrophisms the group of all invertible $2N \times 2N$
matrices (under usual matrix multiplication) of the form
\begin{equation}\label{equ:ele-Q(n)}
g=\begin{pmatrix}
S&S'\\
-S' &S\\
\end{pmatrix},
\end{equation}
where $S$ is an $N\times N$ matrix with entries in $\ev A$ and $S'$
is an $N\times N$ matrix with entries in $\od A$. The morphism
$G(f): G(A) \ra G(B)$ associated to a morphism $f: A \ra B$ is given
entry-wise on elements of the from (\ref{equ:ele-Q(n)}). The
underlying purely even group is isomorphic to $GL(n)$, which can be
defined to be the functor that associates $A \in \salg$ to the group
of invertible matrices of the form (\ref{equ:ele-Q(n)}) with $S'=0$.

The Lie superalgebra $\g=\mathfrak{q}(N)=\text{Lie}(Q(N))$ consists of
all matrices of the form
\begin{equation}\label{equ:ele-q(n)}
X=\begin{pmatrix}
S&S'\\
S' &S\\
\end{pmatrix},
\end{equation}
where $S$ and $S'$ are $N \times N$ matrices over $\mathbb{C}$, and
such an element is even if $S'=0$ or odd if $S=0$. The
multiplication $[.,.]$ is defined by the supercommutator of
matrices. If we let $P=\begin{pmatrix}0&I_N\\-I_N&0 \end{pmatrix}
\in \gl(N|N)$, which is the Lie superalgebra of $(N|N)\times (N|N)$
blocked matrices, then $\g$ is the (super)centralizer of $P$ in
$\gl(N|N)$. The Lie superalgebra $\g$ admits an odd nondegenerate
$\g$-invariant (super)symmetric bilinear form, which is given by
\[
(x,y):= \text{otr}(xy) \text{ for }x, y \in \g,
\]
where $xy$ denotes the matrix product, and $\text{otr}$ denotes the
odd trace given by
\[
\text{otr}\begin{pmatrix}
A&B\\
B&A\\
\end{pmatrix} =\text{trace}(B).
\]
For $1\leq i,j \leq N$, write $e^{\bar{0}}_{i,j}$ (resp. $e^{\bar{1}}_{i,j}$) the element
in $\g$ with $1$ on the $(i,j)$-entry of $S$ (resp. $S'$) and $0$
elsewhere. Define a linear map
\[
\Pi : \g \ra \g,
\]
which interchanges $e^{\bar{0}}_{i,j}$ and $e^{\bar{1}}_{ij}$ for all $1\leq i,j \leq N$.

\subsection{The good $\Z$-gradings}\label{sec:gg}
By a $\Z$-grading of $\g$ we always mean a $\Z$-grading
\[
\Gamma:  \g=\oplus_{j \in \Z} \g_j,
\]
of $\g$ as a Lie superalgebra which is compatible with the
$\Z_2$-grading, i.e. $\g_j=\g_{j,\bar{0}} \oplus \g_{j,\bar{1}}$,
and $[\g_i, \g_j] \subseteq \g_{i+j}$ for all $i, j\in \Z$.

For $k \in \Z$, we shall denote $\g_{>k}=\oplus_{j>k}\g_j$.
Similarly, we define $\g_{\geq k}$, $\g_{< k}$, $\g_{\leq k}$, and
$\g_{\neq k}$.

Let $\chi \in \ev \g^*$ be a nilpotent linear functional and we
always regard $\chi \in \g^*$ by setting $\chi(\od \g) = 0$. Denote
the centralizer of $\chi$ in $\g$ by $\g_{\chi} = \g_{\chi,\bar 0} +
\g_{\chi,\bar 1}$, where $\g_{\chi,i} = \{ y \in \g_i \vert \;
\chi([y, \g]) = 0 \}$ for $i \in \Z_2$.

Such a grading $\Gamma$ is called {\em good} for $\chi$ if the
center $\mathfrak{z}(\g)=\mathbb{C}\cdot I_{2n}$ of $\g$ is contained
in $\g_0$ and if it satisfies the following two conditions.
\begin{equation}\label{equ:gg1}
\chi(\g_{\neq -2})=0
\end{equation}
\begin{equation}\label{equ:gg2}
\g_\chi \subseteq \g_{\geq 0}.
\end{equation}

\begin{lemma}\label{lem:grading-inner}
For any $\Z$-grading $\Gamma: \g=\oplus_{j \in \Z}\g_j$ such that
$\mathfrak{z}(\g) \subseteq \g_0$, there exists a semisimple element
$h_\Gamma \in [\ev\g,\ev\g]$ such that $\Gamma$ coincides with the
eigenspace decomposition of $\text{ad}\,h_\Gamma$, i.e. $\g_j=
\{x\in \g\vert\;[h_\Gamma,x]=jx\}$.
\end{lemma}
\begin{proof}
The degree operator $\partial :\g \ra \g$ which sends $x \ra jx$ for
$x \in \g_j$ is a derivation of $\g$, hence a derivation of $\ev\g$.
Since we require $\mathfrak{z}(\g)\subseteq \g_0$, the grading on
$\ev \g$ is given by $\text{ad}\,h_\Gamma$ for some semisimple
element $h_\Gamma \in [\ev\g, \ev\g]$.

Write $c=I_{2n} \in \ev\g$, and let $C=\Pi(c)$. We claim that $C \in
\g_0$. Indeed, write $C=\sum_j C_j$, with $C_j \in \g_j$. Then for
any $k$
\[
\partial([C, C_k])= \sum_jj[C_j,C_k] +k[C,C_k]=\sum_j(j+k)[C_j,C_k];
\]
On the other hand, we have $\text{ad}\,h_\Gamma(C)=0$ and so $[h_\Gamma, C_k]=0$ for any $k$. Thus,
\[
\partial([C,C_k]) = [h_\Gamma, [C, C_k]]=[[h_\Gamma,C], C_k] + [C,[h_\Gamma, C_k]]=0
\]
It follows that
\[
(j+k)[C_j,C_k]=0, \quad \text{for all }k,j.
\]
In particular, we have
\[
[C_j,C_k]=0, \quad \text{ when }j+k \neq 0.
\]
Now suppose $C_k \neq 0$ for some $k\neq 0$. Then
\[
[C,C_k]=\sum_j [C_j, C_k]=[C_{-k}, C_k]=[C_k, C_{-k}].
\]
A similar calculation shows
\[
[C,C_{-k}] =[C_k, C_{-k}].
\]
It follows that $[C, C_k-C_{-k}]=0$. But this can happen if and only
if $C_k=C_{-k}=0$. Thus we have $C \in \g_0$, as desired.

Now let $X \in \od \g$ such that $x =\Pi(X)$ lies in $\g_i$ for some
$i \in \Z$. Write $X=\sum_j X_j$ with $X_j \in \g_j$. First we have
\[
\partial([X, C])=\partial(2x)=i(2x)=i([X,C])=\sum_j i[X_j, C].
\]
On the other hand, since $C \in \g_0$, we have
\[
\partial([X,C])=\sum_j j[X_j, C].
\]
It follows that
\[
i[X_j, C]=j[X_j,C] \quad \text{ for all }j.
\]
This can only be possible when $X_j=0$ for $j \neq i$. Hence $X=X_i$
and $\partial(X)=\text{ad}\, h_\Gamma(X)$. The lemma thus follows.
\end{proof}

Let $E$ be the element in $\od \g$ defined by the relation
$\chi=(E,.)$, and let $e=\Pi(E)$. The defining condition
(\ref{equ:gg1}) and (\ref{equ:gg2}) are easily seen to be equivalent
to the following
\begin{equation*}\tag{\ref{equ:gg1}'}
E \in \g_2;
\end{equation*}
\begin{equation*}\tag{\ref{equ:gg2}'}
\g_E \subseteq \g_{\geq 0}.
\end{equation*}

It also follows from Lemma~\ref{lem:grading-inner} that
\begin{equation}\label{equ:gg3}
(\g_i, \g_j)=0 \quad \text{unless }i+j=0.
\end{equation}

Lemma~\ref{lem:grading-inner} tells us we need not to distinguish a $\Z$-grading of $\g$
and a $\Z$-grading for its even part $\ev\g=\gl(N)$. Given a
$\Z$-grading of $\gl(N)$, we define a $\Z$-grading on $\g$ by given
$\od \g$, which is an adjoint copy of $\gl(N)$, the same graded
structure but with odd parity; and any $\Z$-grading of $\g$ is
obtained this way. Next we show that that we need not to distinguish a good $\Z$-grading of $\g$
and a good grading of $\gl(N)$. In order to do that, we briefly recall the classification of good gradings for $\ev\g=\mathfrak{gl}(N)$ using pyramids. For more details on this matter, please see \cite[Section~4]{EK} and \cite[Section~6]{Wan}.

Given a partition $\la =(0<p_1 \leq p_2 \leq \ldots \leq p_n)$ of $N$, we
construct a combinatorial object, called pyramids (of shape
$\la$) as follows.

We start with a (lowest) row of $p_n$ boxes of size $2$ units
by $2$ units, with column numbers $1-p_n, 3-p_n, \cdots, p_n
-1$. Then, we add a (second to last) row of $p_{n-1}$ boxes on top of the lowest row. The rule is: keep the stair shape with permissible shifts by
integer units. Then we continue the process with the same rule until we have added all $n$ rows of boxes. In the example of $\la=(2,2,3)$,
we obtain three pyramids below (where the column numbers are
also indicated).
$$
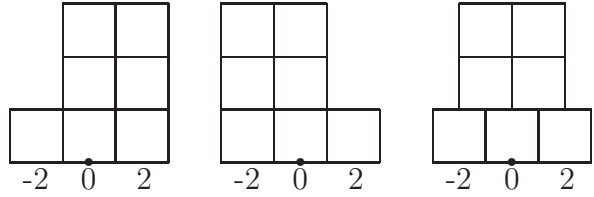
\begin{picture}(220,70)
%\la=(3,3,2)

\put(0,10){\line(1,0){60}} \put(0,30){\line(1,0){60}}
\put(20,50){\line(1,0){40}} \put(20,70){\line(1,0){40}}
\put(60,10){\line(0,1){60}} \put(0,10){\line(0,1){20}}
\put(20,10){\line(0,1){40}} \put(40,10){\line(0,1){60}}
\put(20,50){\line(0,1){20}}
 \put(30,10){\circle*{3}}
\put(5,0){-2} \put(27,0){0} \put(48,0){2}

\put(80,10){\line(1,0){60}} \put(80,30){\line(1,0){60}}
\put(80,50){\line(1,0){40}} \put(80,70){\line(1,0){40}}
\put(140,10){\line(0,1){20}} \put(80,10){\line(0,1){60}}
\put(100,10){\line(0,1){60}} \put(120,10){\line(0,1){60}}
 \put(110,10){\circle*{3}}
\put(85,0){-2} \put(107,0){0} \put(128,0){2}

\put(160,10){\line(1,0){60}} \put(160,30){\line(1,0){60}}
\put(170,50){\line(1,0){40}} \put(170,70){\line(1,0){40}}
\put(220,10){\line(0,1){20}} \put(160,10){\line(0,1){20}}
\put(180,10){\line(0,1){20}} \put(200,10){\line(0,1){20}}
\put(170,30){\line(0,1){40}} \put(190,30){\line(0,1){40}}
\put(210,30){\line(0,1){40}} \put(190,10){\circle*{3}}
\put(165,0){-2} \put(187,0){0} \put(208,0){2}
\end{picture}
$$

Given a pyramid $P$ of shape $\la$, let us fix a labeling by
numbers $\{1,2,\ldots, N\}$  of the $N$ boxes in $P$. A convenient
choice is to label downward from left to right in an increasing
order.

Let $\g=\mathfrak{q}(N)$ with natural module $V$. Write $\{v^{\bar
0}_i=v_i \vert\; i=1,\ldots,N\}$ (resp. $\{v^{\bar 1}_j \vert \;
j=1,\ldots, N\}$) the standard basis for $\ev V=\mathbb{C}^N$ (resp.
$\od V=\mathbb{C}^N$), where $\bar 0, \bar 1 \in \Z_2$. Let $e^P$ be the nilpotent element in $\mathfrak{g}_{\bar{0}}=\mathfrak{gl}(N)$ which sends a vector $v^{\epsilon}_i$ to
$v^{\epsilon}_{L(i)}$ where $\epsilon \in\Z_2$ and $L(i)$ denotes the label to the
left of $i$ in the labelled pyramid $P$ (by convention
$v^\epsilon_{L(i)} =0$ whenever $ L(i)$ is not defined).
Note that $e^P$ has $J_\la$ as its Jordan form.

A $\Z$-grading $\Gamma^P$ of $\ev\g=\mathfrak{gl}(N)$ is determined by letting $\deg (e^\epsilon_{i,j})= \text{col}(j)-\text{col}(i)$, where $\text{col}(i)$ denotes the column number of the box labelled by $i$ in $P$. According to \cite{EK}, all good $\Z$-grading of $\ev\g=\gl(N)$ are obtained this way.

\begin{lemma}\label{lem:grading-good-gl-q}
A $\Z$-grading $\Gamma : \g =\oplus_{j\in\Z}\g_j$ is good for $\chi$
if and only if when restricted to $\ev\g$, $\Gamma$ is a good
grading for $e$ of $\ev\g$.
\end{lemma}
\begin{proof}
Note that a good grading for $\chi$ certainly gives a good
$\Z$-grading for $e$ in $\ev\g$.

Now let $\Gamma$ be a $\Z$-grading of $\g$, such that $e \in \g_2$
and $\g_{e,\bar{0}} \subseteq \g_{\geq 0}$. Then it is clear from
Lemma~\ref{lem:grading-inner} that $E \in \g_2$. We need to show
that $\g_E \subseteq \g_{\geq 0}$. First note that
$\g_{E,\bar{0}}=\g_{e,\bar{0}} \subseteq \g_{\geq 0}$ by assumption.
Secondly, note that $\g_{E,\bar{1}} \subseteq \g_{\geq 0}$ is
equivalent to $\Pi(\g_{E,\bar{1}}) \subseteq \g_{\geq 0}$. But this
is to show that the set of elements in $\ev\g=\gl(N)$ which are
anticommutive with $e$ are non-negatively graded in $\Gamma$. We may assume, without loss of generality,
that the restriction of $\Gamma$ on $\ev\g$ is given by a pyramid
$P$ of shape $\la = (p_1 \leq p_2 \leq \ldots \leq p_n)$ and $e = e^P$. Assume the last boxes on each row are numbered with $t_1,\ldots,t_n$ respectively.

An element $\hat{z}$ that
is anticommutative with $e$ in $\gl(N)$ is determined by the values
$\hat{z}(v_{t_i})$ for $1 \leq i \leq n$, since $\hat{z}(e^kv_{t_i})=(-1)^ke^k\hat{z}(v_{t_i})$.
Consider $\hat{z}_{j,i;k_j} \in \gl(N)$ ($1\leq i,j\leq n$), which anticommute with $e$, such that $\hat{z}_{j,i,k_j}(v_{t_i})=e^{k_j}v_{t_j}$ and $\hat{z}_{j,i;k_j}(v_{{t_{i'}}})=0$ for $1 \leq i' \leq n$, $i' \neq i$.
Then since
$e^{p_i+k_j}v_{t_j}=(-1)^{p_i}\hat{z}_{j,i;k_j}(e^{p_i}v_{t_i})=0$ and $0
\leq k_j<p_j$, $k_j$ has to satisfy the inequality:
\[
p_j >k_j \geq \max(0,p_j-p_i),
\]
and this is sufficient for $\hat{z}_{j,i,k_j}$ to be well-defined
and anticommutative with $e$. For each pair $i,j$, there are
$\min(p_i,p_j)=p_j-\max(0,p_j-p_i)$ choices for such
$\hat{z}_{j,i;k_j}$. So there are in total $\sum_{1\leq i,j\leq
n}\min(p_i,p_j)$ of such $\hat{z}_{j,i;j_k}$ and they are
linearly independent. These elements form a basis of the set of
matrices anticommuting with $e$ since we know from
\cite[Proposition~4.1]{WZ2} that $\dim \g_{E,\bar{1}}=\sum_{1\leq
i,j\leq n}\min(p_i,p_j)$. The elements $\hat{z}_{j,i,k_j}$ are
manifestly homogeneous and non-negatively graded in $\Gamma$
restricted to $\ev\g=\gl(N)$, the lemma thus follows.
\end{proof}

A good grading has following further properties.
\begin{proposition}\label{prop:gg-properties}
Let $\Gamma: \g=\oplus_j \g_j$ be a good grading for a nilpotent
linear functional $\chi=(E,.)$ of $\g$. Then we have
\begin{equation}\label{equ:gg4}
\text{ad}\,E: \g_j \ra \g_{j+2} \text{ is injective for } j\leq -1,
\end{equation}
\begin{equation}\label{equ:gg5}
\text{ad}\, E: \g_j \ra \g_{j+2} \text{ is surjective for } j \geq
-1,
\end{equation}
\begin{equation}\label{equ:gg6}
\sudim \g_E =\sudim \g_0 + \sudim \g_1.
\end{equation}
\end{proposition}
\begin{proof}
The proof is the same as for the Lie algebra case, thus will be
omitted. (see e.g. \cite[Section~2.5]{Wan})
\end{proof}

\section{Finite $W$-superalgebras for queer Lie superalgebras}\label{sec:defW}
\subsection{Definition of $W$-superalgebras}

Fix a nilpotent linear functional $\chi \in \ev\g^*$ and let $\Gamma: \g=\oplus_{i\in\Z}\g_i$ be a good grading for $\chi$. Let $E\in\od\g$ be determined by $\chi(-)=(E,-)$. By
Proposition~\ref{prop:gg-properties}, $\text{ad}\,E:\g_{-1}\ra\g_1$
is bijective. Thus there is a non-degenerate symplectic
(respectively symmetric) bilinear form $\langle \cdot, \cdot\rangle$
on $\g_{-1,{\bar 0}}$ (respectively $\g_{-1,{\bar 1}}$) given by
\begin{equation*}%\label{ssform}
\langle x, y \rangle := (E, [x, y]) = \chi([x, y]).
\end{equation*}
In other words, the above defines an {\em even} non-degenerate
skew-supersymmetric bilinear form $\langle \cdot, \cdot\rangle$ on
$\g_{-1}$. Fix a ($\Z_2$-graded) isotropic subspace
$\mathfrak{l}=\ev{\mathfrak{l}} \oplus \od{\mathfrak{l}}$ of
$\g_{-1}$ with respect to $\langle.,.\rangle$, and let
$\mathfrak{l}'=\{x \in \g_{-1}\vert\; \langle x,
\mathfrak{l}\rangle=0\}$. We have $\mathfrak{l} \subseteq
\mathfrak{l}'$. Define nilpotent subalgebras $\mathfrak{m} \subseteq
\m'$ as follows:
\[
\m=\mathfrak{l}\oplus \bigoplus_{i \leq -2}\g_i \quad
\text{and}\quad \m'=\mathfrak{l}'\oplus \bigoplus_{i \leq -2}\g_i.
\]
The linear functional $\chi$ restricts to a character on $\m$.
Denote by $\mathbb{C}_\chi$ the corresponding $1$-dimensional
representation of $\m$, and define the {\em generalized Gelfand-Graev module}
\[
Q_{\mathfrak{l}}=U(\g)
\otimes_{U(\m)}\mathbb{C}_\chi=U(\g)/I_{\mathfrak{l}}
\]
be the induced $U(\g)$-module, where $I_{\mathfrak{l}}$ denotes the
left ideal of $U(\g)$ generated by $a-\chi(a)$ for all
$\Z_2$-homogeneous $a\in \m$. Then $I_{\mathfrak{l}}$ is
$\text{ad}\,\m'$-invariant, thus there is an induced
$\text{ad}\,\m'$-action on $Q_{\mathfrak{l}}$. Following \cite{GG}
(which is in turn a generalization of \cite{Pr2}, cf. also
\cite{WZ2}), we define the $W$-superalgebra associated to the
isotropic subspace $\mathfrak{l}$ to be
\[
\mathcal{W}_{\mathfrak{l}}:=(U(\g)/I_{\mathfrak{l}})^{\text{ad}\,\m'}
\cong \{\bar{y} \in U(\g)/I_{\mathfrak{l}} \vert\; [b,y] \in
I_{\mathfrak{l}}, \;\forall b\in \m'\},
\]
where $\bar{y}$ stands for the coset of $y\in U(\g)$ in $U(\g)/I_{\mathfrak{l}}$. The multiplication is given by
$\bar{y}_1\bar{y}_2=\overline{y_1y_2}$ for $\bar{y}_1, \bar{y}_2 \in
\mathcal{W}_{\mathfrak{l}}$.

%\begin{theorem}\label{thm:W-indep}
%The superalgebra $\mathcal{W}_{\mathfrak{l}}$ is independent of the
%choice of an isotropic subspace $\mathfrak{l} \subseteq \g_{-1}$.
%\end{theorem}

%The rest of this section is devoted to prove this theorem.

\subsection{} Write $e=\Pi E$. It follows from
Lemma~\ref{lem:grading-good-gl-q} and \cite[Lemma~1.1]{EK} that
there exist $h \in \g_{0,\bar{0}}$ and $f \in \g_{-2,\bar{0}}$ such
that $\{e,h,f\}$ form an $\mathfrak{sl}(2)$-triple, which will be called the $\Gamma$-graded $\mathfrak{sl}(2)$-triple. Put $H=\Pi h$
and $F=\Pi f$. %Then we have corresponding properties
%(\ref{equ:gg1}')-(\ref{equ:gg2}') and
%(\ref{equ:gg3})-(\ref{equ:gg6}) for the element $F$.

Given a $\Z_2$-graded subspace $M$ of $\g$, we let $M^\perp=\{x \in
\g \vert\; (x,v)=0,\;\forall v\in M\}$, and let
$M^{*,\perp}=\{\xi\in \g^* \vert\; \xi(v)=0,\; \forall v \in M\}$.

\begin{lemma}\label{lem:m-perp}
We have $\m^\perp=[\m', E] \oplus \Pi \g_F$.
\end{lemma}
\begin{proof}
The lemma follows from the four facts below.
\begin{itemize}
\item[(i)] $\m^\perp \supseteq \Pi \g_F$. This follows form $\m^\perp \supseteq \g_{\leq
0}$ by (\ref{equ:gg3}) and the $F$-counterpart to (\ref{equ:gg2}')
which says that $\g_F \subseteq \g_{\leq 0}$ and thus $\Pi \g_F
\subseteq \g_{\leq 0}$ according to Lemma~\ref{lem:grading-inner}.
\item[(ii)] $\m^\perp \supseteq [\m',E]$. This follows from the
computation:\\ $(m,[m',E])$ $=$ $([m,m'],E)=0$ for $m \in \m$ and
$m'\in \m'$.
\item[(iii)] $[\m',E] \cap \Pi\g_F=0$. It suffices to show that
\begin{equation}\label{equ:imagepi}
\text{im}(\text{ad}\,E) \cap \Pi
\g_F=0.
\end{equation}
There are actually two identities, one in each
$\Z_2$-parity. The odd part of (\ref{equ:imagepi}) is equivalent to, as operators
on $\ev\g=\gl(N)$, $\text{im}(\text{ad}\,e) \cap \text{ker}(\ad\,f)
=0$, which is a result of $\mathfrak{sl}(2)$-representation theory.
If we define two operators on $\ev\g$ as follows
\begin{align*}
\text{ad}^+e: & \gl(N) \ra \gl(N),& x \mapsto ex+xe;\\
\text{ad}^+f: & \gl(N) \ra \gl(N), &x \mapsto fx+xf.
\end{align*}
Then we can check that $\{\text{ad}^+e, \text{ad}\,h, \text{ad}^+f
\}$ form an $\mathfrak{sl}(2)$-triple in $\text{End}_{\mathbb{C}}(\gl(n))$. Now
observe that the even part of (\ref{equ:imagepi}) is equivalent of saying
$\text{im}(\text{ad}^+e) \cap \text{ker}(\text{ad}^+ f)=0$, which is
now also a consequence of $\mathfrak{sl}(2)$-theory.
\item[(iv)] $\sudim \m^\perp = \sudim \Pi\m' +\sudim \g_0 + \sudim \g_{-1}=\sudim [\m', E]+\sudim
\Pi\g_F$. This follows from the (parity-shifting) bijection $\m' \ra
[\m', E]$, $x \mapsto [x,E]$, by (\ref{equ:gg4}) and the $F$-counter
part to (\ref{equ:gg6}).
\end{itemize}
\end{proof}

\subsection{$\mathbb{C}^*$-actions}
We describe some $\mathbb{C}^*$-actions on some affine superschemes; such
an action is equivalent to specifying a (right)
$\mathbb{C}[t,t^{-1}]$-comodule structure on the coordinate ring of
the affine superscheme.

The dual space $\g^*$ of $\g$ carries an induced grading
$\g^*=\oplus_j \g^*_j$ from $\text{ad}^*h$, where $h$ comes from the
$\Gamma$-graded $\mathfrak{sl}(2)$-triple. Define a
$\mathbb{C}^*$-action on $\g^*$ as follows.
\begin{align*}
\rho:\g^* & \ra \g^* \otimes \mathbb{C}[t,t^{-1}],\\
\xi &\mapsto \xi \otimes t^{2-j},
\end{align*}
where $\xi \in \g^*_j$. Note that this is equivalent to specifying a
comodule structure on $\mathbb{C}[\g^*]=S(\g)$ such that the comodule structure on $\mathbb{C}[\g^*]$ is a superalgebra homomorphism.

Let $\text{ad}^*: \g \ra \text{End}_{\mathbb{C}}(\g^*)$ denote the coadjoint action of $\g$ on $\g^*$.
The closed subsuperschemes $\chi+ (\text{ker}\,\text{ad}^*F)\circ \Pi$
and $\chi+ \m^{*,\perp}$ are stable under the action $\rho$ (meaning
$\mathbb{C}[\chi+ (\text{ker}\,\text{ad}^*F)\circ \Pi]$ and
$\mathbb{C}[\chi+ \m^{*,\perp}]$ are $\mathbb{C}[t,t^{-1}]$-subcomodules of
$\mathbb{C}[\g^*]$). Also the action on the underlying even variety
of $\chi+ \m^{*,\perp}$ (resp. $\chi+ (\text{ker}\,\text{ad}^*F)\circ
\Pi$) is contracting with the fixed point $\chi$.

The embedding of the $\Gamma$-graded $\mathfrak{sl}(2)$-triple in
$\g$ exponentiates to a rational homomorphism $\tilde{\gamma}:
\text{SL}_2\ra G_{\text{ev}} \ra G$. Define a $\mathbb{C}^*$-action
$\gamma$ on $G$ by conjugation by
$\tilde{\gamma}(\text{diag}(t^{-1},t))$, for $t \in \mathbb{C}^*$.

Now let $M'$ be the closed subgroup of $G$ whose Lie superalgebra is
$\m'$. The
$\mathbb{C}^*$-action on $M' \times (\chi+
(\text{ker}\,\text{ad}^*F)\circ \Pi)$ is defined to be $\gamma$ on
the first factor and $\rho$ on the second. Then this action on the
underlying even variety is also contracting with the fixed point
$(1, \chi)$.

\subsection{} Denote by $\kappa : \g \ra \g^*$ the isomorphism
induced by the non-degenerate bilinear form $(.,.)$. Following the
terminology of Gan and Ginzburg \cite{GG}, we will call
\[
\mathcal{S}:=\chi + (\text{ker} \,\text{ad}^*F) \circ \Pi \cong
\kappa(E + \Pi \g_F)
\]
the Slodowy slice (through $\chi$).

Using the isomorphism of vector superspaces $\kappa$,
Lemma~\ref{lem:m-perp} actually translates to the fact that the
differential map of the the coadjoint action map
\[
\alpha: M' \times (\chi+(\text{ker} \,\text{ad}^*F) \circ \Pi) \ra
\chi + \m^{*,\perp}
\]
is an isomorphism between the tangent spaces at the points
$(1,\chi)$ and $\chi$, i.e.
\[
\text{ad}^*\m'(\chi) \oplus (\text{ker}\,\text{ad}^*F)\circ
\Pi=\m^{*,\perp}.
\]

\begin{lemma}\label{lem:iso-var}
The coadjoint action map
\[
\alpha: M' \times (\chi+(\text{ker} \,\text{ad}^*F) \circ \Pi) \ra
\chi + \m^{*,\perp}
\]
is an isomorphism of affine superschemes.
\end{lemma}
\begin{proof}
We sketch a proof here following \cite[(7.7)]{Gin}.

First of all, note that $\alpha$ is $\mathbb{C}^*$-equivariant.

Write $X_1$ for $M' \times (\chi+(\text{ker} \,\text{ad}^*F) \circ
\Pi)$ and $X_2$ for $\chi + \m^{*,\perp}$. Denote $\chi_X \in
\mathbb{C}[[t]]$ the formal character of the coordinate ring of a
$\mathbb{C}^*$-superscheme $X$.

Let $x_i$ denote the unique $\mathbb{C}^*$-fixed point in the
underlying even variety of $X_i$, and let $T_i=T_{x_i}X_i$ be the
tangent space of $X_i$ at $x_i$. By Lemma~\ref{lem:m-perp} and the
equivariance of $\alpha$, we have $\chi_{T_1}=\chi_{T_2}$.

Introduce an adic filtration of $\mathbb{C}[X_i]$ by powers of the
maximal ideal of the point $x_i$. Then by definition, we have
$\mathbb{C}[T_i] \cong \text{gr}\,\mathbb{C}[X_i]$ as superalgebras.
Since $x_i$ is a fixed point under the $\mathbb{C}^*$-action, the
isomorphism is actually an isomorphism of
$\mathbb{C}[t,t^{-1}]$-comodules. It follows that
$\chi_{X_1}=\chi_{T_1}=\chi_{T_2}=\chi_{X_2}$. As a result,
$\text{gr}\,\mathbb{C}[X_2] \cong \text{gr}\,\mathbb{C}[X_1]$. It
follows from the isomorphism of the graded version that $\alpha^*:
\mathbb{C}[X_2] \ra \mathbb{C}[X_1]$ is injective. Finally,
$\alpha^*$ has to be surjective also since $\chi_{X_1}=\chi_{X_2}$.
\end{proof}

\subsection{The Kazhdan grading and filtration}
Let $\{U_j(\g)\}$ be the standard PBW filtration on $U(\g)$. The
action of $\text{ad}\,h$ induces a grading on each $U_j(\g)$ by
\[
U_j(\g)_i=\{x\in U_j(\g) \vert\; \text{ad}\,h(x)=ix\}.
\]
The Kazhdan filtration on $U(\g)$ is defined by letting the
\[
F_l^KU(\g)=\sum_{i+2k\leq l} U_k(\g)_i.
\]
The associated grading on $\text{gr}^KU(\g)$ will be called Kazhdan
grading. We also define the Kazhdan gradings for $\g$ and $S(\g)$ in
a similar fashion. Similar to the Lie algebra case, the Kazhdan
filtration has the following properties.
\begin{itemize}
\item[(1)] The canonical map $\text{gr}^KU(\g) \ra
S(\g)=\mathbb{C}[\g^*]$ is an isomorphism of graded commutative
superalgebras.

\item[(2)] There is a Kazhdan filtration $\{F_k^KQ_{\mathfrak{l}}\}$
on $Q_{\mathfrak{l}}=U(\g)/I_{\mathfrak{l}}$ induced from $U(\g)$.
And the filtration satisfies $F_k^KQ_{\mathfrak{l}}=0$ unless $k \geq
0$.

\item[(3)]
$\text{gr}^KQ_{\mathfrak{l}}=\text{gr}^KU(\g)/\text{gr}^KI_{\mathfrak{l}}$
is a commutative $\Z_+$-graded superalgebra.

\item[(4)] The ideal $\text{gr}^K I_{\mathfrak{l}}$ in $\text{gr}^KU(\g)=\mathbb{C}[\g^*]$ can be identified with the ideal of polynomial functions on $\g^*$ which vanish on $\chi+\m^{*,\perp}$. The canonical map $\text{gr}^K Q_{\mathfrak{l}} \ra \mathbb{C}[\chi+\m^{*,\perp}]$ is an algebra isomorphism.

\item[(5)] There is an induced Kazhdan filtration $\{F^K \mathcal{W}_{\mathfrak{l}}\}$ on the subspace $\mathcal{W}_{\mathfrak{l}}$ of $Q_{\mathfrak{l}}$ such that $F_j^K\mathcal{W}_{\mathfrak{l}}=0$ unless $j\geq 0$.
\end{itemize}

Thus we have the following diagram
\begin{equation}\label{equ:diagram}
\xymatrix{
  \text{gr}^KU(\g) \ar[d]_{} \ar@{=}[r]^{} & \mathbb{C}[\g^*] \ar[d]^{} \\
  \text{gr}^KQ_{\mathfrak{l}}  \ar@{=}[r]^{} & \mathbb{C}[\chi+\m^{*,\perp}] \ar[d]^{} \\
  \text{gr}^K\mathcal{W}_{\mathfrak{l}}\ar[u]_{} \ar[r]^{\nu} & \mathbb{C}[\mathcal{S}]   }
\end{equation}
where
$\nu:\text{gr}^K\mathcal{W}_{\mathfrak{l}} \ra
\mathbb{C}[\mathcal{S}]$ is the composition of the three natural
maps
\[
\text{gr}^K\mathcal{W}_{\mathfrak{l}} \ra \text{gr}^KQ_{\mathfrak{l}}
\ra \mathbb{C}[\chi+\m^{*,\perp}] \ra \mathbb{C}[\mathcal{S}].
\]

\subsection{Definition of $\mathcal{W}_{\mathfrak{l}}$: Independence of $\mathfrak{l}$ and the grading}

\begin{lemma}\label{lem:A-dR-super}
We have
\[
H^i(\mathfrak{m}',\mathbb{C}[M'])=\delta_{i,0}\mathbb{C},
\]
where $H^i(\mathfrak{m}',\mathbb{C}[M'])$ denotes the $i$th Lie
superalgebra cohomology of $\mathfrak{m}'$ with coefficient in
$\mathbb{C}[M']$.
\end{lemma}
\begin{proof}
In the same way of proving Lemma~\ref{lem:iso-var}, we can show that
$M'$ is isomorphic to the affine superspace
$\chi+\text{ad}^*\m'(\chi)$. Then the coordinate superalgebra
$\mathbb{C}[M']$ is isomorphic to a polynomial superalgebra
$\mathbb{C}[x_1, \ldots, x_s;\xi_1,\ldots,\xi_s]$. Note that, the
Lie superalgebra cohomology $H^i(\mathfrak{m}',\mathbb{C}[M'])$, $i
\geq 0$, can be computed from the cohomology of the cochain complex
\begin{equation}\label{equ:LSU-cohom}
\text{Hom}(\wedge^\bullet \text{Der}_{\mathbb{C}}(\mathbb{C}[M']), \mathbb{C}[M']),
\end{equation}
where $\text{Der}_{\mathbb{C}}(\mathbb{C}[M'])$ is the algebra of
derivations of $\mathbb{C}[M']$. After identifying $\mathbb{C}[M']$
with the polynomial superalgebra $\mathbb{C}[x_1, \ldots,
x_s;\xi_1,\ldots,\xi_s]$, the cochain complex (\ref{equ:LSU-cohom})
becomes
\begin{equation}\label{equ:Koszul_complex}
\text{Hom}(\wedge^\bullet \text{Der}_{\mathbb{C}}(\mathbb{C}[x_1,
\ldots, x_s;\xi_1,\ldots,\xi_s]), \mathbb{C}[x_1, \ldots,
x_s;\xi_1,\ldots,\xi_s]).
\end{equation}
The cohomology of (\ref{equ:Koszul_complex}), which can be computed
in the same way as in [Kos, proof of Theorem~4.6], is exactly as
desired.
\end{proof}

Recall that $\m'$ is graded with respect to the grading $\Gamma$. We
view $U(\g)$ and $Q_{\mathfrak{l}}$ as $\m'$-modules via the adjoint
$\m'$-action. Then $U(\g)$ and $Q_{\mathfrak{l}}$ are Kazhdan
filtered $\m'$-modules and the canonical map $p: U(\g)\ra
Q_{\mathfrak{l}}$ is $\m'$-equivariant. Thus, $\text{gr}^KU(\g)$ and
$\text{gr}^KQ_{\mathfrak{l}}$ are Kazhdan graded $\m'$-modules, and
the map $\text{gr}^Kp: \text{gr}^KU(\g)\ra \text{gr}^KQ_{\mathfrak{l}}$ is
also $\m'$-equivariant.

By definition
$\mathcal{W}_{\mathfrak{l}}=H^0(\m',Q_{\mathfrak{l}})$, the $0$th
Lie superalgebra cohomology of $\m'$ with coefficient in
$Q_{\mathfrak{l}}$.

\begin{theorem}\label{thm:isom-nu}
The map $\nu: \text{gr}^K\mathcal{W}_{\mathfrak{l}} \ra
\mathbb{C}[\mathcal{S}]$ is an isomorphism of graded superalgebras; it
%: \text{gr}^K\mathcal{W}_{\mathfrak{l}} \ra
%\mathbb{C}[\mathcal{S}]$
is equal to the composite $\nu_2\nu_1$:
\[
\text{gr}^KH^0(\m', Q_{\mathfrak{l}}) \xrightarrow{\nu_1} H^0(\m',
\text{gr}^KQ_{\mathfrak{l}}) \xrightarrow{\nu_2} \mathbb{C}[\mathcal{S}]
\]
where both $\nu_1$ and $\nu_2$ are isomorphisms. Moreover,
\[
H^i(\m',Q_{\mathfrak{l}})=H^i(\m',\text{gr}^KQ_{\mathfrak{l}})=0,
\]
for all $i>0$.
\end{theorem}
%\begin{proof}
% We reformulate...
%\end{proof}
%\begin{theorem}\label{thm:composition-nu}

%\end{theorem}
\begin{proof}
The proof is a straightforward generalization of
\cite[Section~5]{GG}. We make a sketch here.

By Lemma~\ref{lem:iso-var}, we have isomorphisms of vector spaces
\[
\text{gr}^KQ_{\mathfrak{l}} \cong \mathbb{C}[\chi+\m^{*,\perp}]\cong
\mathbb{C}[M']\otimes \mathbb{C}[\mathcal{S}].
\]
These isomorphisms are actually on the level of $\m'$-modules, where
the $\m'$-module structure on the third space comes from the
$\m'$-adjoint action on its first tensor factor $\mathbb{C}[M']$.
The statements involving $\text{gr}^KQ_{\mathfrak{l}}$ in the
theorem follows from Lemma~\ref{lem:A-dR-super}.

Note that the Kazhdan filtration on $Q_{\mathfrak{l}}$ has no
negative-degree component. Also, $\m'$ is a negatively graded
subalgebra of $\g$ with respective to $\Gamma$, so its dual
${\m'}^*$ is positively graded. Write this graded decomposition as
${\m'}^*=\oplus_{i\geq 1}{\m'}^*_{i}$.

Consider the standard cochain complex for computing the Lie
supalgebra $\m'$-cohomology of $Q_{\mathfrak{l}}$:
\begin{equation}\label{equ:coh}
0\ra Q_{\mathfrak{l}} \ra {\m'}^* \otimes Q_{\mathfrak{l}} \ra
\cdots \ra \wedge^k {\m'}^*\otimes Q_{\mathfrak{l}} \ra \cdots
\end{equation}
A filtration on $\wedge^k{\m'}^*\otimes Q_{\mathfrak{l}}$ is defined
by letting $F_p(\wedge^k{\m'}^*\otimes Q_{\mathfrak{l}})$ be the
subspace of $\wedge^k{\m'}^*\otimes Q_{\mathfrak{l}}$ spanned by
$(x_1\wedge \ldots \wedge x_k)\otimes v$, for all
$x_1\in{\m'}^*_{i_1}$,\ldots, $x_k \in {\m'}^*_{i_k}$ and $v \in
F_j^KQ_{\mathfrak{l}}$ such that $i_1 +\cdots +i_k+j \leq p$. This
defines a filtered complex structure on (\ref{equ:coh}). Taking the
associated graded complex of (\ref{equ:coh}) gives the standard
cochain complex for computing the $\m'$-cohomology of
$\text{gr}^KQ_{\mathfrak{l}}$.

Now consider the spectral sequence with
\[
E^{p,q}_0=F_p(\wedge^{p+q}{\m'}^*\otimes
Q_{\mathfrak{l}})/F_{p-1}(\wedge^{p+q}{\m'}^*\otimes
Q_{\mathfrak{l}}).
\]
Then $E^{p,q}_1=H^{p+q}(\m', \text{gr}^K_pQ_{\mathfrak{l}})$, and
the spectral sequence converges to $E^{p,q}_{\infty}$ $=$
$F_pH^{p+q}(\m',Q_{\mathfrak{l}})/F_{p-1}H^{p+q}(\m',Q_{\mathfrak{l}})$.
The rest of the theorem follows from this and the parts about
$\text{gr}^KQ_{\mathfrak{l}}$ established above.
%The proof, which uses Lemma~\ref{lem:A-dR-super}, is verbatim to the
%Lie algebra case as in \cite[Section~5]{GG}.
\end{proof}

\begin{theorem}\label{thm:isotropic}
The superalgebras $\mathcal{W}_{\mathfrak{l}}$ are all isomorphic
for different choices of isotropic subspaces $\mathfrak{l} \subset
\g_{-1}$.
\end{theorem}
\begin{proof}
The proof is the same as the Lie algebra case in \cite[Section~5.5]{GG}.
\end{proof}

From the classification of good gradings of $\g=\mathfrak{q}(N)$ in
Lemma~\ref{lem:grading-good-gl-q}, we also have the following.
\begin{theorem}
The $W$-superalgebra associated to any two good gradings $\Gamma$
and $\Gamma'$ for $\chi$ are isomorphic.
\end{theorem}
\begin{proof}
The proof, which uses Theorem~\ref{thm:isotropic} and
\cite[Theorem~2]{BG}, is the same as proof of \cite[Theorem~1]{BG},
thus will be omitted.
\end{proof}

\subsection{Skryabin equivalence}
As we have established the independence of the $W$-superalgebras from the
choices of the isotropic subspaces $\mathfrak{l}$ and the good gradings, we will change the
notations for the generalized Gelfand-Graev module and the
$W$-superalgebra etc. to $Q_\chi, \mathcal W_\chi$ etc., to emphasize the
crucial dependence on $\chi$. In the remainder of this section, we
shall fix a nilpotent linear functional $\chi$ and a {\em
Lagrangian} subspace $\mathfrak{l}$ of $\g_{-1}$ once for all.

A $\g$-module $L$ is called a {\em Whittaker module} if $a
-\chi(a)$, $\forall a \in \m$, acts on $L$ locally nilpotently. A
{\em Whittaker vector} in a Whittaker $\g$-module $L$ is a vector $x
\in L$ which satisfies $(a -\chi(a)) x=0, \forall a \in \m. $

Let $\whmod$ be the category of finitely generated Whittaker
$\g$-modules with even homomorphisms.

Denote the subspace of all Whittaker vectors in $L$ by
$$
\wh (L) = \{v\in L \mid (a -\chi(a)) v=0, \forall a \in \m\}.
$$
Recall that $\mathcal{W}_{\chi} =(U(\g)/I_\chi)^{\text{ad}\, \m}$,
and we denote by $\bar{y} \in U(\g)/I_\chi$ the coset associated to
$y \in U(\g)$.
\begin{lemma}
\begin{enumerate}
\item Given a Whittaker $\g$-module $L$ with an action map
$\rho$, $ \wh (L)$ is naturally a $\mathcal{W}_{\chi}$-module by
letting $\bar{y}. v =\rho(y)v$ for $v\in \wh (L)$ and $\bar{y} \in
\mathcal{W}_{\chi}$.

\item For $M \in \Wmod$, $Q_\chi \otimes_{\mathcal{W}_{\chi}}M$ is a Whittaker
$\g$-module by letting
$$\quad y. (q \otimes v) =(y.q) \otimes v, \quad \text{ for }
y \in U(\g), q \in Q_\chi =U(\g)/I_\chi, v \in V.
$$
\end{enumerate}
\end{lemma}
\begin{proof}
The proof is straightforward and is the same as the Lie algebra case. (see e.g. proof of \cite[Lemma~35]{Wan}).
\end{proof}

Let $\Wmod$ be the category of finitely generated $\mathcal
W_\chi$-modules with even homomorphisms. We define the {\em
Whittaker functor}
\begin{equation*}
\wh:  \whmod \longrightarrow \Wmod, \qquad L \mapsto \wh (L).
\end{equation*}
We define another functor
$$
Q_\chi\otimes_{\mathcal{W}_{\chi}} -:   \Wmod   \longrightarrow
\whmod, \qquad M \mapsto Q_\chi\otimes_{\mathcal{W}_{\chi}}M.
$$

\begin{theorem}\label{thm:skry}
The functor $Q_\chi \otimes_{\mathcal{W}_{\chi}} -:   \Wmod
\longrightarrow \whmod$ is an equivalence of categories, with $\wh:
\whmod \longrightarrow \Wmod$ as its quasi-inverse.
\end{theorem}
\begin{proof}
The proof, like the proof of Theorem~\ref{thm:isom-nu}, is the same as the Lie algebra case \cite[Theorem~6.1]{GG}, thus
will be omitted.
\end{proof}

\begin{remark}
We note here that Skryabin's approach \cite{Skr} can be generalized to our setting without difficulties. This allows us to remove the ``finite generation'' condition in Theorem~\ref{thm:skry}.
\end{remark}

\begin{remark}\label{rem:basicclassical}
Some of the constructions here in Section~\ref{sec:defW} admit natural generalizations in basic classical Lie superalgebras.
Let $\mathfrak{s}$ be a basic classical Lie superalgebra or the Lie superalgebra $\gl(m|n)$, and let $\chi\in \ev{\mathfrak{s}}^*$ be a nilpotent linear functional. We can define good gradings for $\chi$ and $\mathfrak{s}$ using (\ref{equ:gg1}) and (\ref{equ:gg2}). We start with a good grading for $\chi$, then there is a super-skewsymmetric bilinear form on the degree $-1$ component $\mathfrak{s}_{-1}$. We can pick an isotropic subspace $\mathfrak{l}$ and define the corresponding generalized Gelfand-Graev module and the $W$-superalgebra as we did in this section. One major difference from the queer Lie superalgebra is that $\dim \mathfrak{s}_{-1,\bar{1}}$ might be odd (as noted in \cite{WZ1}). When $\dim \mathfrak{s}_{-1,\bar{1}}$ is odd, Lemma~\ref{lem:A-dR-super} may go wrong. Nonetheless, when $\dim \mathfrak{s}_{-1,\bar{1}}$ is even, the proofs in this section all carry through to establish the fact that the definition of the $W$-superalgebra is independent of the choice of isotropic subspaces as well as the Skryabin equivalence.
\end{remark}

%\subsection{Examples}
%\begin{example}
%When $\chi=0$, $\g_\chi=\g$. Then $\g=\g_0$ gives a good grading. We
%will have $\m=\m'=0$, and $I_\chi=0$. Thus
%$\mathcal{W}_{\chi}=U(\g)$ in this case.
%\end{example}

%\begin{example}
%When $\chi$ is regular nilpotent...
%\end{example}

%\noindent {\bf To Do List.}
%\begin{itemize}
%\item Existence of closed subgroup $M'$.
%\item proof of Lemma~\ref{lem:A-dR-super}.
%\item description of $\mathcal{W}_\chi$ when $\chi$ is regular
%nilpotent, subregular...
%\item Skryabin equivalence when $\dim \g_{-1,{\bar{1}}}$ is odd.
%\item good gradings for basic classical.
%\end{itemize}

%%%%%%%%%%%%%%%%%%%%%%%%%
%\subsection{Basis Theorem}

\end{document}